\title{On-line Ramsey numbers}
\author{David Conlon\thanks{St John's College, Cambridge, CB2 1TP, United Kingdom. E-mail: {\tt D.Conlon@dpmms.cam.ac.uk}. Supported by a Junior Research Fellowship at St John's College, Cambridge.}}
\date{}

\documentclass[11pt]{article}
\usepackage{amsfonts,amssymb,amsmath,latexsym}

\newenvironment{proof}
      {\medskip\noindent{\bf Proof.}\hspace{1mm}}
      {\hfill$\Box$\medskip}

\def\qed{\ifvmode\mbox{ }\else\unskip\fi\hskip 1em plus 10fill$\Box$}

\newtheorem{theorem}{Theorem}[section]

\newtheorem{lemma}[theorem]{Lemma}

\newtheorem{problem}[theorem]{Problem}

\begin{document}

\maketitle

\begin{abstract}
Consider the following game between two players, Builder and Painter. Builder draws edges one at a time and Painter colours them, in either red or blue, as each appears. Builder's aim is to force Painter to draw a monochromatic copy of a fixed graph $G$. The minimum number of edges which Builder must draw, regardless of Painter's strategy, in order to guarantee that this happens is known as the on-line Ramsey number $\tilde{r}(G)$ of $G$. Our main result, relating to the conjecture that $\tilde{r}(K_t) = o(\binom{r(t)}{2})$,
is that there exists a constant $c > 1$ such that 
\[\tilde{r}(K_t) \leq c^{-t} \binom{r(t)}{2}\]
for infinitely many values of $t$. We also prove a more specific upper bound for this number, showing that there exists a constant $c$ such that
\[\tilde{r}(K_t) \leq t^{-c \frac{\log t}{\log \log t}} 4^t.\]
Finally, we prove a new upper bound for the on-line Ramsey number of the complete bipartite graph $K_{t,t}$.
\end{abstract}

\section{Introduction}

Given a graph $G$, the Ramsey number $r(G)$ of $G$ is the smallest number $n$ such that, in any two-colouring of the edges of the complete graph $K_n$, there is guaranteed to be a monochromatic copy of $G$. In the specific case where $G = K_t$, we write $r(t)$ for the Ramsey number.

The size Ramsey number $\hat{r}(G)$, introduced by Erd\H{o}s, Faudree, Rousseau and Schelp \cite{EFRS78}, is the smallest number $m$ such that there exists a graph $H$ with $m$ edges which is Ramsey with respect to $G$, that is, no matter how one colours the edges of $H$ in two colours, red and blue say, one will always find a monochromatic copy of $G$.

For cliques, the size Ramsey number, written simply as $\hat{r}(t)$, is just the same as the number of edges in the complete graph on $r(t)$ vertices \cite{EFRS78}, that is
\[\hat{r}(t) = \binom{r(t)}{2},\]
a result attributed to Chv\'{a}tal.

More surprisingly, Beck proved \cite{B83} that there exists a constant $c$ such that
\[\hat{r}(P_n) \leq cn,\]
where $P_n$ is the path with $n$ vertices. Similar results have been proven when $G$ is a cycle with $n$ vertices \cite{HKL95} and when $G$ is a tree with $n$ vertices and a fixed maximum degree \cite{FP87}.

It is only recently that Kohayakawa, R\"{o}dl, Szemer\'edi and Schacht \cite{KRSS09} have managed to extend Beck's result to all graphs of bounded maximum degree. Their result states that for every natural number $\Delta$ there exists a constant $c$ such that the size Ramsey number of any graph $H$ with $n$ vertices and maximum degree $\Delta$ satisfies
\[\hat{r}(H) \leq c n^{2 - 1/\Delta} \log^{1/\Delta} n.\]
Even though this result is highly impressive, a very large gap still remains in this problem, the best lower bound, due to R\"{o}dl and Szemer\'edi \cite{RS00}, only being that there exists a graph on $n$ vertices with maximum degree 3 for which
\[\hat{r}(H) \geq c n \log^{1/60} n.\]
This does, however, rule out the possibility, raised by Beck and Erd\H{o}s (see \cite{CG98}), that for every $\Delta$ there exists a constant $c$ such that, for any graph $H$ with $n$ vertices and maximum degree $\Delta$, $\hat{r}(H) \leq c n$. 

An on-line variant of the size Ramsey number was introduced independently by Beck \cite{B93} and Kurek and Ruci\'nski \cite{KR05}. It is best described in terms of a game between two players, known as Builder and Painter. Builder draws a series of edges and, as each edge appears, Painter colours it, in either red or blue. Builder's goal is to force Painter to draw a monochromatic copy of a fixed graph $G$. The least number of edges which Builder must draw in order to force Painter, regardless of their strategy, to draw such a monochromatic copy of $G$ is known as the size Ramsey number $\tilde{r}(G)$ of $G$. As usual, $\tilde{r}(t)$ is the same as $\tilde{r}(K_t)$.

The basic conjecture in the field, attributed by Kurek and Ruci\'nski \cite{KR05} to R\"{o}dl and reiterated by several others \cite{FKRRT03, GHK04, GKP08, KK09, P308}, is to show that 
\[\lim_{t \rightarrow \infty} \frac{\tilde{r}(t)}{\hat{r}(t)} = 0,\]
that is, $\tilde{r}(t) = o(\binom{r(t)}{2})$. Our main result approaches this conjecture by showing that there is a subsequence $\{t_1, t_2, \cdots\}$ of the integers such that 
$\lim_{i \rightarrow \infty} \frac{\tilde{r}(t_i)}{\hat{r}(t_i)} = 0$.
In fact, we do even better, showing that there are infinitely many values of $t$ for which the on-line Ramsey number is exponentially smaller than the size Ramsey number.

\begin{theorem} \label{IntroMain}
There exists a constant $c > 1$ such that, for infinitely many $t$,
\[\tilde{r}(t) \leq c^{-t} \binom{r(t)}{2}.\]
\end{theorem}

We will also prove a more specific bound for the on-line Ramsey number of the complete graph, bringing it in line with the recent bound for diagonal Ramsey numbers due to the author \cite{C09}. The former bound, due to Pra\l at \cite{P308}, was $\tilde{r}(t) \leq c \frac{4^t}{\sqrt{t}}$.

\begin{theorem} \label{IntroSpec}
There exists a constant $c$ such that
\[\tilde{r}(t) \leq t^{-c \frac{\log t}{\log \log t}} 4^t.\]
\end{theorem}

Finally, we consider the on-line Ramsey number of complete bipartite graphs. For the size Ramsey number of $K_{t,t}$, the bounds \cite{NR78, EFRS78, ER93} are surprisingly close, being  
\[\frac{1}{60} t^2 2^t < \hat{r}(K_{t,t}) < \frac{3}{2} t^3 2^t.\]
Nevertheless, an on-line Builder can force Painter to draw a monochromatic $K_{t,t}$ even sooner.

\begin{theorem} \label{IntroBip}
There exists a constant $c$ such that
\[\tilde{r}(K_{t,t}) \leq c 2^t t^{5/2} \log^{1/2} t.\]
\end{theorem}

Throughout the paper, we have, for the sake of presentation, systematically omitted floor and ceiling signs. We have also made no serious attempt to optimise any of the constants in the statements and proofs. In particular, the constant $c$ that arises in our proof of Theorem \ref{IntroMain} is certainly not optimal, but we felt that any serious attempt to optimise it would serve only to mask the main ideas.

\section{Complete graphs}

It is natural to expect that the on-line Ramsey number of almost every graph $G$ is decidedly smaller than its size Ramsey number. However, while it has been proved \cite{KR05} that $\tilde{r}(s,t) = o(\binom{r(s,t)}{2})$ as $t \rightarrow \infty$ and $s$ remains fixed, a similar result is not yet known for the diagonal Ramsey number $r(t)$. In this section we make some progress in this direction by proving that there is a subsequence $\{t_1, t_2, \cdots \}$ of the integers for which 
\[\lim_{i \rightarrow \infty} \frac{\tilde{r}(t_i)}{\hat{r}(t_i)} = 0.\]
The reason this result only works for a subsequence is that in the course of the proof we need a result of the form
\[\frac{r(t+1)}{r(t)} \geq c > 1.\]
Unfortunately, though such a result is almost certainly true, it seems incredibly difficult to prove. We can, however, exploit the fact that it is true on average to prove our theorem. 

\begin{theorem} \label{Main}
For infinitely many $t$,
\[\tilde{r}(t) \leq 1.001^{-t} \binom{r(t)}{2}.\]
\end{theorem}

\begin{proof}
Let $\alpha = 0.01$. Consider the following strategy. To begin, Builder draws $n-1$ edges emanating from a single vertex $v_1$. Painter must paint at least $(n-1)/2$ of these the same colour. Let $V_1$ be the neighborhood of $v_1$ in that colour. We also define a string $s$ in terms of the colours chosen. We initialise this string by writing $s = R$ if the majority colour was red and $s=B$ if it were blue.

Suppose now that we are looking at a set $V_i$. We choose any given vertex $v_{i+1}$ and draw all the neighbours of $v_{i+1}$ within $V_i$. If, in the string $s$, there are more $R$s than $B$s, we choose $V_{i+1}$ to be the neighborhood of $v_{i+1}$ in red if $|V_{i+1}| \geq (1 - \alpha) (|V_i| - 1)$ and the neighborhood in blue otherwise. Similarly, if there are more blues than reds, one chooses $V_{i+1}$ to be the neighborhood of $v_{i+1}$ in blue if and only if $|V_{i+1}| \geq (1 - \alpha) (|V_i| - 1)$. If the number of reds and blues are the same, then we follow whichever has more neighbours. The string $s$ then has whichever colour we followed appended to it. If, for example, our string were $a_1 \cdots a_i$, with each $a_j=R$ or $B$, and we followed red, the new string would be $a_1 \cdots a_i R$.

Let $\mu = 0.99$ and $\nu = 0.01$. The process stops when the string $s$ contains either $\mu t$ $R$s, $\mu t$ $B$s or $\nu t$ $R$s and $\nu t$ $B$s. Suppose, at that stage, that we have chosen $m$ vertices and $m$ neighborhoods. Builder's strategy now is to fill in a complete subgraph of $V_m$ of size $p$ equal to the maximum of  $r((1-\mu)t, t)$ and $r((1-\nu)t)$. It is easy to see that Painter will be forced to draw a complete graph in one colour or the other. Suppose, for example, that our string contains $\mu t$ blues. Since $p \geq r((1-\mu)t, t)$, $V_m$ contains either a red clique of size $t$, in which case we're done, or a blue clique of size $(1-\mu) t$. This latter clique may be appended to the $\mu t$ vertices which correspond to $B$ in the string to form a blue clique of size $t$. The other cases follow similarly. 

We need now to estimate the number of edges that Builder has to draw. In order to guarantee that the process works, we need to start with an $n$ which will guarantee that $|V_m| \geq p$. If we made the most expensive choice at each point as we were choosing the $v_i$, we may have to choose $n$ to be as large as $(2/\alpha)^{\nu t} (1 - \alpha)^{- \mu t} p$. Since $m \leq (\mu + \nu) t$, it is then elementary to see that the number of edges Builder draws is at most
\[m n + \binom{p}{2} \leq t (2/\alpha)^{\nu t} (1 - \alpha)^{-t} p + \binom{p}{2}.\]

To estimate the value of this expression, we must first understand something about $p$. By the choice of $\mu$,
\begin{eqnarray*}
r((1-\mu)t, t) = r(0.01 t, t) & \leq & \binom{1.01 t}{t}
\leq \left(\frac{1.01 e t}{0.01 t}\right)^{0.01 t} \\
& \leq &  1.06^t \leq 1.25^{-t} r(t),
\end{eqnarray*}
since $r(t) \geq \sqrt{2}^t$.  

On the other hand, there must be infinitely many values $t$ for which
\[r((1-\nu) t) = r(0.99t) \leq 1.001^{-t} r(t).\]
Suppose otherwise. Then there exists $t_0$ such that, for all $t \geq t_0$,
\[\frac{r(t)}{r(0.99t)} \leq 1.001^t.\]
By telescoping, this would imply that
\begin{eqnarray*}
r(0.99^{-A} t_0) & \leq & (1.001)^{(0.99^{-1} + \cdots + 0.99^{-A}) t_0} r(t_0) \\
& \leq & (1.001)^{100 (0.99)^{-A} t_0} r(t_0).
\end{eqnarray*}
If we rewrite this equation, with $t = 0.99^{-A} t_0$ and $C = r(t_0)$, we see that this would imply
\[r(t) \leq C (1.001)^{100 t} \leq C (1.106)^t,\]
which is plainly a contradiction for large $t$.

We may therefore conclude that $p \leq 1.001^{-t} r(t)$ infinitely often. At such values of $t$ we have that the number of edges Builder must draw to force Painter to draw a monochromatic $K_t$ is less than
\begin{eqnarray*}
t (2/\alpha)^{\nu t} (1 - \alpha)^{-t} p + \binom{p}{2} & \leq & t (200)^{0.01t} (0.99)^{-t} p + \binom{p}{2}\\
& \leq & t (1.066)^t p + \binom{p}{2} \leq \frac{r(t) - 1}{4} p + \binom{p}{2} \\
& \leq & 1.001^{-t} \binom{r(t)}{2},
\end{eqnarray*}
provided that $t$ is sufficiently large.
\end{proof}

Let $\tilde{r}_q (t)$ be the $q$-colour on-line Ramsey number of $K_t$. That is, Builder and Painter play the usual game, but Painter is now allowed to use $q$ colours. It is worth noting that, for small $q$, the same method can be used to give some results on $\tilde{r}_q (t)$, but it does not allow one to show unreservedly that $\tilde{r}_q (t) \leq c^{-t} \binom{r_q(t)}{2}$, where $r_q (t)$ is the ordinary $q$-colour Ramsey number of $K_t$. The following theorem serves as an illustration, though we stress that the condition on $r_3(t)$ is not best possible.

\begin{theorem}
There exists a constant $c$ such that, for infinitely many $t$, if $r_3(t) \geq 4^t$ then
\[\tilde{r}_3 (t) \leq c^{-t} \binom{r_3(t)}{2}.\] 
\end{theorem}

Having looked at the ratio of the on-line Ramsey number to the size Ramsey number, we now turn our attention to proving a more specific bound on $\tilde{r}(t)$. If one looks at the standard proof of Ramsey's theorem, it is easy to prove a bound of the form $\tilde{r}(t) \leq 4^t$. By being more careful with this approach Pra\l at proved \cite{P308} that $\frac{3}{8\sqrt{\pi}} \frac{4^t}{\sqrt{t}}$. Applying the following recent improvement, due to the author \cite{C09}, on the upper bound for classical Ramsey numbers, we can improve Pra\l at's bound much further. 

\begin{theorem} \label{DiagonalRamsey}
There exists a constant $C$ such that, for $\frac{1}{2} l \leq k \leq 2l$,
\[r(k,l) \leq k^{-C \frac{\log k}{\log \log k}} \binom{k+l}{k}.\] 
\end{theorem}

Note that the numbers $\frac{1}{2}$ and $2$ are not important and may be replaced by any other fixed positive constants.

\begin{theorem} \label{Specifics}
There exists a constant $c$ such that
\[\tilde{r}(t) \leq t^{-c \frac{\log t}{\log \log t}} 4^t.\]
\end{theorem}

\begin{proof}
We follow the usual strategy. To begin, choose a vertex $v_1$ and draw $n-1$ edges emanating from it. This vertex must have at least $(n-1)/2$ neighbours in (at least) one of red or blue. We fix such a colour, calling it $C_1$. Let $V_1$ be the vertex set consisting of the neighbours of $v_1$ in $C_1$. Now, given $V_i$, choose any element $v_{i+1}$ of $V_i$ and draw all the edges connecting $v_{i+1}$ to other vertices in $V_i$. $v_{i+1}$ must have at least $(|V_i|-1)/2$ neighbours in colour $C_i$, say. We let $V_{i+1}$ be the set of neighbours of $v_{i+1}$ in $C_i$.

We stop the process after $m = \frac{3}{2}t$ steps. Let $p$ be the maximum value of $r(t-a, t-b)$ taken over all $a$ and $b$ with $a + b = \frac{3}{2}t$. We claim that if $|V_m| \geq p$ then Builder has a winning strategy. To see this, suppose that $a$ of the $C_i$ are red and $b$ of them are blue. Since $|V_m| \geq r(t-a, t-b)$, Builder may draw a complete subgraph in $V_m$ of size $r(t-a, t-b)$. Once Painter has coloured this, $V_m$ must contain either a red $K_{t-a}$ or a blue $K_{t-b}$. Connecting these back to either the $a$ vertices whose associated colour is red or the $b$ vertices whose associated colour is blue would imply a monochromatic $K_t$. Choosing $n = 2^{3t/2} p$, $V_m$ will necessarily have the required size $|V_m| \geq p$.

How many edges does Builder draw in this process? For each vertex $v_i$, $1 \leq i \leq m$, he draws at most $n$ edges, and within $V_m$ he draws at most $\binom{p}{2}$ edges. Therefore,
\begin{eqnarray*}
\tilde{r}(t) & \leq & \sum_{i=1}^m n + \binom{p}{2}\\
& \leq & m n + p^2\\
& \leq & t 2^{3t/2 + 1} p + p^2.
\end{eqnarray*}

We now need a bound for $p$. For the second part of this sum it is sufficient to note that, for all $a$ and $b$ with $a + b = \frac{3}{2} t$,
\[r(t-a,t-b) \leq 2^{(t-a) + (t-b)} \leq 2^{2t - (a + b)} \leq 2^{t/2}.\]
For the first part we must use Theorem \ref{DiagonalRamsey}. When $\frac{1}{2} (t-b) \leq t - a \leq 2(t-b)$, we must have that $t - a \geq t/8$. Otherwise, we would have $t - b \leq t/4$ and therefore $a + b \geq 13t/8$. Therefore, by Theorem \ref{DiagonalRamsey}, if $\frac{1}{2}(t - b) \leq t - a \leq 2(t-b)$, there exists $c$ such that 
\begin{eqnarray*}
r(t-a,t-b) & \leq & t^{-c \frac{\log t}{\log \log t}} \binom{2t - (a+b)}{t-a}\\
& \leq & t^{-c \frac{\log t}{\log \log t}} 2^{2t - (a+b)} = t^{-c \frac{\log t}{\log \log t}} 2^{t/2}.
\end{eqnarray*}
If, on the other hand, $t - a \leq \frac{1}{2}(t - b)$ (or, by symmetry, $t- b \leq \frac{1}{2} (t-a)$), then
\begin{eqnarray*}
r(t-a, t-b) & \leq & \binom{2t - (a+b)}{t-a}
= \prod_{i=1}^{t-b} \left(1 + \frac{t-a}{t-b-i+1}\right)\\
& = & \binom{(t - a) + 3(t-b)/4}{t-a} \prod_{i=1}^{(t-b)/4} \left(1 + \frac{t-a}{t-b-i+1}\right)\\
& \leq & 2^{(t-a) + 3(t-b)/4} \prod_{i=1}^{(t-b)/4} \left(1 + \frac{t-a}{3(t-b)/4} \right)\\
& \leq & 2^{(t-a) + 3(t-b)/4} (5/3)^{(t-b)/4}\\
& \leq & t^{-c \frac{\log t}{\log \log t}} 2^{2t - (a+b)} = t^{-c \frac{\log t}{\log \log t}} 2^{t/2}.
\end{eqnarray*}
It therefore follows that
\begin{eqnarray*}
\tilde{r}(t) & \leq & t 2^{3t/2 + 1} p + p^2\\
& \leq & 2 t^{-c \frac{\log t}{\log \log t} + 1} 4^t + 2^t. 
\end{eqnarray*}
The result follows from a slight adjustment of the constant $c$.    
\end{proof}

For $\tilde{r}_q (t)$, an obvious upper bound is $q^{qt}$, though this can be improved to $c_q t^{-(q-1)/2} q^{qt}$. As things stand, this agrees, up to the constant, with the best known bound for $r_q(t)$. However, using the method above, one can prove that any improvement on the bound for $r_q(t)$ implies a corresponding improvement for $\tilde{r}_q (t)$. 

\begin{theorem} \label{qSpec}
Let $q$ be an integer. Suppose that $f: \mathbb{N} \rightarrow \mathbb{R}$ is a function such that if $k_1 \geq k_2 \geq \cdots \geq k_q \geq \frac{1}{2} k_1$, then 
\[r(k_1, k_2, \cdots, k_q) \leq f(k_1) q^{k_1 + k_2 + \cdots + k_q}.\]
Then there exists a constant $c > 0$ such that
\[\tilde{r}_q (t) \leq f(ct) q^{qt}.\]
\end{theorem}

The one extra component needed in the proof is the following formula, which generalises the standard Erd\H{o}s-Szekeres bound for two-colour Ramsey numbers by using multinomial coefficients.
\[r(k_1, k_2, \cdots, k_q) \leq \binom{k_1 + k_2 + \cdots + k_q}{k_1, k_2, \cdots, k_q}.\]
Note that Theorem \ref{Specifics} is the special case of Theorem \ref{qSpec} where $q = 2$ and $f(x) = x^{-C \frac{\log x}{\log \log x}}$. 

\section{Complete bipartite graphs}

The method we will use in this section is closely related to that used by the author \cite{C08} in improving the bipartite Ramsey number of $K_{t,t}$ from $2^t t$ to, roughly, $2^t \log t$. It involves a two-step strategy for Builder. In the first step, using around $2^t$ edges, he forces Painter to draw a complete monochromatic, say blue, bipartite graph with $t - 2 \log t$ vertices on one side and $t^2$ on the other. In the second step Builder draws all the edges from this set $U$ of size $t^2$ to a new set of size $2^t \sqrt{t \log t}$. This forces Painter to draw either a red $K_{t,t}$, in which case we are done, or a complete bipartite graph in blue with a set $U' \subset U$ of $t$ vertices on one side and $2 \log t$ vertices on the other. In the latter case, we may add these $2 \log t$ vertices to the $t - 2 \log t$ already joined to $U$ to form a blue $K_{t,t}$. This then yields the result. 

In order to prove our results we will need to make repeated use of the following fundamental result about bipartite graphs, due essentially to K\H{o}v\'ari, S\'os and Tur\'an \cite{KST54}.

\begin{lemma} \label{KST}
Let $G$ be a bipartite graph with parts $A$ and $B$ and density at least $\epsilon$. Let $m = |A|$ and $n = |B|$. Then, provided that
\[m \geq 2 \epsilon^{-1} r^2 \mbox{ and } n \geq 2 \epsilon^{-r} s,\]
the graph contains a complete bipartite subgraph with $r$ vertices from $A$ and $s$ vertices from $B$. 
\end{lemma}

\begin{proof}
Suppose otherwise. If there are no complete subgraphs having $r$ vertices in $A$ and $s$ in $B$, there can be at most $\binom{|A|}{r} s$ pairs $(A',v)$, where $A' \subset A$ has size $r$ and every element in $A'$ is joined to $v$. On the other hand, the number of such pairs is equal to the sum $\sum_{v \in B} \binom{d(v)}{r}$. By convexity,
\begin{eqnarray*}
\sum_{v \in B} \binom{d(v)}{r} & \geq & n \binom{\frac{1}{n} \sum_{v \in B} d(v)}{r}
\geq n \binom{\epsilon m}{r}\\
& = & n \frac{(\epsilon m)^r}{r!} \prod_{i=1}^{r-1} \left(1 - \frac{i}{\epsilon m} \right).  
\end{eqnarray*}
Now, since $\frac{r}{\epsilon m} \leq \frac{1}{2}$ and, for $x \leq \frac{1}{2}$, $e^{-2x} \leq 1 - x$,
\begin{eqnarray*}
\prod_{i=1}^{r-1} \left(1 - \frac{i}{\epsilon m} \right) & \geq & \prod_{i=1}^{r-1} e^{-2i/\epsilon m}\\
& \geq & e^{-r^2/\epsilon m} \geq \frac{1}{2}. 
\end{eqnarray*}
Therefore, comparing our two different ways of counting the number of pairs $(A', v)$, we find that
\[\frac{n}{2} \frac{(\epsilon m)^r}{r!} \leq \binom{m}{r} s \leq \frac{m^r}{r!} s.\]
This implies that
\[n \leq 2 \epsilon^{-r} s,\]
a contradiction.
\end{proof}

As one might expect, given a natural number $q$ and a graph $G$, $\tilde{r}_q(G)$ is the $q$-colour on-line Ramsey number of $G$. Since our result extends easily to the $q$-colour case, we give the proof in that level of generality. 
 
\begin{theorem}
\[\tilde{r}_q (K_{t,t}) \leq 48 q^{t+2} t^{3-1/q} \log^{1/q} t,\]
where $\log$ is taken to the base $q$.
\end{theorem}

\begin{proof}
Builder begins by building a complete bipartite graph between sets $M$ and $N$, $M$ having size $6 q^{t+1} \log t$ and $N$ having size $2 q t^2$. Suppose, without loss of generality, that after Painter has coloured these edges the blue graph has density at least $1/q$. By Lemma \ref{KST} with $A = N$, $B = M$, $\epsilon = 1/q$, $r = t - 2 \log t$ and $s = 3 q t^2$, since
\[|M| \geq 6 q^{t+1} \log t = 2 q^{t - 2 \log t} (3 q t^2),\]
there must be a complete blue bipartite graph $H$ with the part in $M$ having size $3 q t^2$ and that in $N$ having size $t - 2 \log t$.

Builder now passes to the second phase, drawing a complete bipartite graph between the vertices $M'$ of $H$ that are in $M$ and a newly created vertex set $N'$ of size $12 q^t t^{1-1/q} \log^{1/q} t$. Let $\epsilon = \frac{(q-1) (\log t - \log\log t)}{q^2 t}$. Once Painter has coloured all the edges either the blue graph will have density at least $\frac{1}{q} - \epsilon$ or the graph in some other colour, red say, will have density at least $\frac{1}{q} + \frac{\epsilon}{q-1}$. Using the facts that $q \epsilon \leq \frac{1}{2}$ for all $t$ and, for $x \leq \frac{1}{2}$, $(1-x)^{-1} \leq 1 + x + 2x^2 \leq e^x + 2 x^2$, we note that
\begin{eqnarray*}
\left(\frac{1}{q} - \epsilon\right)^{-t} & \leq & q^t \left(1 - q \epsilon\right)^{-t}
\leq q^t \left(e^{q \epsilon} + 2 (q \epsilon)^2\right)^t\\
& = & q^t \left(\left(\frac{t}{\log t}\right)^{(q-1)/qt} + 2 \frac{\log^2 t}{t^2}\right)^t\\
& \leq & q^t \left(\frac{t}{\log t}\right)^{1-1/q} \left(1 + 2 \frac{\log^2 t}{t^2}\right)^t \\
& \leq & 3 q^t \left(\frac{t}{\log t}\right)^{1-1/q}. 
\end{eqnarray*}
In the last line we used that $\left(1 + 2 \frac{\log^2 t}{t^2}\right)^t \leq 3$. Similarly,
\[\left(\frac{1}{q} + \frac{\epsilon}{q-1}\right)^{-t} \leq 3 q^t \left(\frac{\log t}{t}\right)^{1/q}.\]
If now the density of the blue edges is greater than $\frac{1}{q} - \epsilon$, an application of Lemma \ref{KST} with $A =M'$, $B=N'$, $r = t$ and $s = 2 \log t$ tells us that there is a complete blue bipartite graph with $t$ vertices in $M'$ and $2 \log t$ vertices in $N'$. Otherwise we would have either $|M'| \leq 3 q t^2$ or
\[|N'| \leq 2 \left(\frac{1}{q} - \epsilon\right)^{-t} 2 \log t \leq 12 q^t t^{1-1/q} \log^{1/q} t.\]
Let $M''$ be the vertices of this complete bipartite graph which lie in $M'$. By construction, $|M''| = t$. Moreover, since the vertices in $M''$ have $t - 2 \log t$ joint neighbours in $N$ and $2 \log t$ joint neighbours in $N'$, we have a blue $K_{t,t}$.

If, on the other hand, one of the other colours, red say, has density greater than $\frac{1}{q} - \frac{\epsilon}{q-1}$, another application of Lemma \ref{KST} would imply that if the red bipartite graph between $M'$ and $N'$ did not contain a complete $K_{t,t}$, then
\[|N'| \leq 2 \left(\frac{1}{q} - \frac{\epsilon}{q-1}\right)^{-t} t \leq 6 q^t t^{1-1/q} \log^{1/q} t.\]
The theorem now follows since
\begin{eqnarray*}
\tilde{r}_q (K_{t,t}) & \leq & |M||N| + |M'||N'|\\ 
& \leq & (6 q^{t+1} \log t)(2 q t^2) + (3 q t^2)(12 q^t t^{1-1/q} \log^{1/q} t)\\
& \leq & 12 q^{t+2} t^2 \log t + 36 q^{t+1} t^{3-1/q} \log^{1/q} t\\
& \leq & 48 q^{t+2} t^{3-1/q} \log^{1/q} t.
\end{eqnarray*}
\end{proof} 
 
\section{Conclusion}

Naturally, the most interesting question in the field is still to show that 
\[\lim_{t \rightarrow \infty} \frac{\tilde{r}(t)}{\hat{r}(t)} = 0,\]
but, given that this is probably extremely difficult, what other questions might one consider?

One possibility is to look at graphs of bounded maximum degree. The simplest case, that of  
determining the on-line Ramsey number of paths, has already been considered by Grytczuk, Kierstead, and Pra\l at \cite{GKP08}, who proved that $\tilde{r}(P_n) \leq 4n-7$. This is a much sharper result than what is known for the ordinary size Ramsey number, where the best known constant is still in the hundreds \cite{B01}. Extensive computational work done by Pra\l at \cite{GKP08, P08, P208} suggests that the constant may be even smaller again. Indeed, for $4 \leq n \leq 9$, the bound $\lfloor \frac{5}{2} n - 5\rfloor$ is correct. It would be very interesting to know whether such a bound holds generally.

For graphs of bounded maximum degree, the best known upper bound is the same as the recent upper bound for the size Ramsey number, that is,
\[\tilde{r}(t) \leq c n^{2 - 1/\Delta} \log^{1/\Delta} n.\]  
On the other hand, we know almost nothing about the lower bound. Even the following question remains open.

\begin{problem}
Given a natural number $\Delta$, does there exist a constant $c$ such that, if $H$ is a graph on $n$ vertices with maximum degree $\Delta$,
\[\tilde{r} (G) \leq cn?\]
\end{problem}

Another direction one can take is to consider the on-line game under the additional restriction that Builder can only draw graphs contained within a given (monotone decreasing) class. The most impressive theorem in this direction, proved by Kierstead, Grytczuk, Ha\l uszczak and Konjevod over two papers \cite{GHK04, KK09}, is that Builder may restrict to graphs of chromatic number at most $t$ and still force Painter to draw a monochromatic $K_t$. The proof of this result is quite intricate and relies upon the analysis of an auxiliary Ramsey game played on hypergraphs.

One beautiful question of this variety, due to Butterfield, Grauman, Kinnersley, Milans, Stocker and West \cite{BGKMSW09}, is to determine whether, given a natural number $\Delta$, there exists $d(\Delta)$ such that Builder can force Painter to draw a monochromatic copy of any graph with maximum degree $\Delta$ by drawing only graphs of maximum degree at most $d(\Delta)$.

\vspace{2mm} \noindent {\bf Acknowledgements.} I would like to thank Jacob Fox for reading carefully through an earlier version of this paper and making several helpful suggestions.


\begin{thebibliography}{}

\bibitem{B83}
{J. Beck,} {On size Ramsey number of paths, trees and cycles I,} {\it J. Graph Theory}
{\bf 7} (1983), 115-–130.

\bibitem{B93}
{J. Beck,} {Achievement games and the probabilistic method,} {\it Combinatorics, Paul Erd\H{o}s is Eighty,} {Bolyai Soc. Math. Stud.} {\bf 1} (1993), 51-–78.

\bibitem{B01}
{B. Bollob\'{a}s,} {\bf Random graphs, Second Edition}, {Cambridge studies in advanced mathematics, vol. 73,} {Cambridge University Press, Cambridge,} (2001).

\bibitem{BGKMSW09}
{J. Buttefield, T. Grauman, B. Kinnersley, K. Milans, C. Stocker and D. West,} {On-line Ramsey theory for bounded degree graphs,} in preparation. 

\bibitem{CG98}
{F. Chung and R. Graham,} {\bf Erd\H{o}s on graphs. His Legacy of Unsolved Problems}, {A K Peters, Ltd., Wellesley, MA,} (1998).

\bibitem{C08}
{D. Conlon,} {A new upper bound for the bipartite Ramsey problem,} {\it J. Graph Theory} {\bf 58} (2008), 351--356.

\bibitem{C09}
{D. Conlon,} {A new upper bound for the diagonal Ramsey problem,} to appear in {\it Annals of Math}. 

\bibitem{EFRS78}
{P. Erd\H{o}s, R.J. Faudree, C.C. Rousseau, and R.H. Schelp,} {The size Ramsey number,} {\it Period. Math. Hungar.} {\bf 9} (1978), 145-–161.

\bibitem{ER93}
{P. Erd\H{o}s and C.C. Rousseau,} {The size Ramsey number of a complete bipartite graph,}
{\it Discrete Math.} {\bf 113} (1993), 259--262.

\bibitem{FKRRT03}
{E. Friedgut, Y. Kohayakawa, V. R\"odl, A. Ruci\'nski, and P.
Tetali,} {Ramsey games against a one-armed bandit,} {\it Combin.
Prob. Comp.} {\bf 12} (2003), 515--545.

\bibitem{FP87}
{J. Friedman and N. Pippenger,} {Expanding graphs contain all small trees,} {\it Combinatorica} 
{\bf 7} (1987), 71--76.

\bibitem{GHK04}
{J.A. Grytczuk, M. Ha\l uszczak, and H.A. Kierstead,} {On-line Ramsey Theory,} {\it Electronic Journal of Combinatorics} { \bf 11} (2004), no. 1, Research Paper 60, 10 pp.

\bibitem{GKP08}
{J.A. Grytczuk, H.A. Kierstead, and P. Pra\l at,} 
{On-line Ramsey Numbers for Paths and Stars,} {\it Discrete Math. Theor. 
Comp. Science} {\bf 10:3} (2008), 63--74.

\bibitem{HKL95}
{P.E. Haxell, Y. Kohayakawa and T. \L uczak,} {The induced size-Ramsey number of cycles,} 
{\it Combin. Prob. Comp.} {\bf 4} (1995), 217--239.

\bibitem{K93}
{Xin Ke,} {The size Ramsey number of trees with bounded degree,} {\it Random Struct. Algorithms} {\bf 4} (1993), 85--97.

\bibitem{KK09}
{H.A. Kierstead and G. Konjevod,} {Coloring number and on-line Ramsey theory for graphs and hypergraphs,} {\it Combinatorica}, accepted.

\bibitem{KRSS09}
{Y. Kohayakawa, V. R\"{o}dl, M. Schacht and E. Szemer\'{e}di,} {Sparse partition universal graphs for graphs of bounded degree,} submitted.

\bibitem{KST54}
{T. K\H{o}v\'ari, V.T. S\'os and P. Tur\'an,} {On a problem of K. Zarankiewicz,} {\it Colloq. Math.} {\bf 3} (1954), 50--57. 

\bibitem{KR05}
{A. Kurek and A. Ruci\'nski,} {Two variants of the size Ramsey
number,} {\it Discuss. Math. Graph Theory} {\bf 25} (2005), 141--149.

\bibitem{NR78}
{J. Ne\v set\v ril and V. R\"{o}dl,} {The structure of critical graphs,} {\it Acta. Math. Acad. Sci. Hungar.} {\bf 32} (1978) 295--300.

\bibitem{P08}
{P. Pra\l at,} {A note on small on-line Ramsey numbers for paths and their generalization,} {\it Australasian Journal of Combinatorics} {\bf 40} (2008), 27-–36.

\bibitem{P208}
{P. Pra\l at,} {A note on off-diagonal small on-line Ramsey numbers for paths,} {\it Ars Combinatoria} 10pp, accepted.

\bibitem{P308}
{P. Pra\l at,} {$\overline{R}(3; 4) = 17$,} {\it Electronic Journal of Combinatorics} {\bf 15} (2008), no. 1, Research Paper 67, 13pp.

\bibitem{RS00}
{V. R\"{o}dl and E. Szemer\'{e}di,} {On size Ramsey numbers of graphs with bounded maximum degree,} {\it Combinatorica} {\bf 20} (2000), 257–-262.


\end{thebibliography}
\end{document}